\documentclass[a4paper]{amsart}
\usepackage{amsthm}
\usepackage{amssymb}
\usepackage{enumerate}
\usepackage{verbatim}
\usepackage{hyperref}

\usepackage{tikz-cd} 

\numberwithin{equation}{section}

\theoremstyle{plain}

\newtheorem{lemma}{Lemma}[section]
\newtheorem{theorem}[lemma]{Theorem}
\newtheorem{proposition}[lemma]{Proposition}
\newtheorem{corollary}[lemma]{Corollary}

\theoremstyle{definition}

\newtheorem{definition}[lemma]{Definition}
\newtheorem{example}[lemma]{Example}

\theoremstyle{remark} 
\newtheorem{remark}[lemma]{Remark}

\newcommand{\Add}{\operatorname{Add}}

\newcommand{\chr}{\operatorname{\Delta}}
\newcommand{\cl}{\operatorname{\mathrm{cl}}}
\newcommand{\coind}{\operatorname{\mathrm{coind}}}

\newcommand{\Coloc}{\operatorname{\mathsf{Coloc}}}
\newcommand{\cosupp}{\operatorname{cosupp}}
\renewcommand{\dim}{\operatorname{dim}}
\newcommand{\End}{\operatorname{End}}
\newcommand{\Ext}{\operatorname{Ext}}

\newcommand{\Hom}{\operatorname{Hom}}

\newcommand{\ind}{\operatorname{\mathrm{ind}}}

\newcommand{\Ker}{\operatorname{Ker}}

\newcommand{\Loc}{\operatorname{\mathsf{Loc}}}
\renewcommand{\mod}{\operatorname{\mathsf{mod}}}
\newcommand{\Mod}{\operatorname{\mathsf{Mod}}}

\newcommand{\Prod}{\operatorname{Prod}}
\newcommand{\Proj}{\operatorname{Proj}}

\newcommand{\res}{\operatorname{res}}

\newcommand{\sHom}{\underline{\Hom}}
\newcommand{\Spec}{\operatorname{Spec}}

\newcommand{\StMod}{\operatorname{\mathsf{StMod}}}
\newcommand{\stmod}{\operatorname{\mathsf{stmod}}}

\newcommand{\supp}{\operatorname{supp}}
\newcommand{\pisupp}{\pi\text{-}\supp}
\newcommand{\picosupp}{\pi\text{-}\cosupp}
\newcommand{\Thick}{\operatorname{\mathsf{Thick}}}

\newcommand{\comp}{\mathop{\circ}}
\newcommand{\col}{\colon}
\newcommand{\ges}{{\scriptscriptstyle\geqslant}}

\newcommand{\kos}[2]{{#1}/\!\!/{#2}}

\newcommand{\op}{\mathrm{op}}

\newcommand{\da}{{\downarrow}}
\newcommand{\lra}{\longrightarrow}

\newcommand{\xra}{\xrightarrow}

\def\mcE{\mathcal{E}}

\def\mcU{\mathcal{U}} 
\def\mcV{\mathcal{V}}

\def\sfc{\mathsf c}

\def\sfA{\mathsf A} 
 
\def\sfC{\mathsf C}

\def\bbP{\mathbb P}
 
\def\bbZ{\mathbb Z}

\newcommand{\fa}{\mathfrak{a}} 

\newcommand{\fm}{\mathfrak{m}} 
\newcommand{\fp}{\mathfrak{p}}
\newcommand{\fq}{\mathfrak{q}}

\newcommand{\gam}{\varGamma} 
\newcommand{\lam}{\varLambda}

\def\Ga1{\operatorname{\mathbb G_{a(1)}}\nolimits}

\def\Tr{\operatorname{Tr}\nolimits}

\title[Colocalising subcategories]{Colocalising subcategories of modules over \\ finite group schemes}

\author[Benson, Iyengar, Krause, and Pevtsova]{Dave Benson, Srikanth
  B. Iyengar, Henning Krause \\ and Julia Pevtsova}

\address{Dave Benson \\ 
Institute of Mathematics\\ 
University of Aberdeen\\ 
King's College\\ 
Aberdeen AB24 3UE\\ 
Scotland U.K.}

\address{Srikanth B. Iyengar\\ 
Department of Mathematics\\
University of Utah\\ 
Salt Lake City, UT 84112\\ 
U.S.A.}

\address{Henning Krause\\ 
Fakult\"at f\"ur Mathematik\\ 
Universit\"at Bielefeld\\ 
33501 Bielefeld\\ 
Germany.}

\address{Julia Pevtsova\\ 
Department of Mathematics\\ 
University of Washington\\ 
Seattle, WA 98195\\ 
U.S.A.}

\begin{document}

\begin{abstract} 
  The Hom closed colocalising subcategories of the stable module
  category of a finite group scheme are classified. This complements
  the classification of the tensor closed localising subcategories in
  our previous work. Both classifications involve $\pi$-points in the
  sense of Friedlander and Pevtsova. We identify for each $\pi$-point
  an endofinite module which both generates the corresponding minimal
  localising subcategory and cogenerates the corresponding minimal
  colocalising subcategory.
\end{abstract}


\date{11th September 2016}

\thanks{SBI was partly supported by NSF grant DMS-1503044 and JP was partly supported by  NSF grants DMS-0953011 and DMS-1501146.}

\maketitle

\section{Introduction}

Let $G$ be a finite group scheme over a field $k$ of positive characteristic. There is a notion of $\pi$-cosupport~\cite{Benson/Iyengar/Krause/Pevtsova:2015a} for any $G$-module $M$, based on the notion of $\pi$-points of $G$ introduced by Friedlander and the fourth author~\cite{Friedlander/Pevtsova:2005a}. The $\pi$-cosupport of $M$, denoted $\picosupp_G(M)$, is a subset of $\Proj H^*(G,k)$.  The main result in this work is a classification of the colocalising subcategories of $\StMod G$, the stable module category of possibly infinite dimensional $G$-modules, in terms of $\pi$-cosupport.

\begin{theorem}
\label{th:costratification} 
Let $G$ be a finite group scheme over a field $k$. Then the assignment
\begin{equation*}
 \sfC\longmapsto \bigcup_{M\in\sfC}\picosupp_G(M)
\end{equation*} induces a bijection between the colocalising
subcategories of $\StMod G$  that are closed under tensor product
with simple $G$-modules
and the subsets of $\Proj H^*(G, k)$.
\end{theorem}

The theorem is proved after Corollary~\ref{co:costratified}. Recall that a \emph{colocalising subcategory} $\sfC$ is a full triangulated subcategory that is closed under set-indexed products. Such a $\sfC$ is closed under tensor product with simple $G$-modules if and only if it is \emph{Hom closed}: If $M$ is in $\sfC$, so is $\Hom_{k}(L,M)$ for any $G$-module $L$.  Theorem~\ref{th:costratification} complements the classification of the tensor closed localising subcategories of $\StMod G$ from \cite{Benson/Iyengar/Krause/Pevtsova:2015b}. Combining them gives a remarkable bijection:

\begin{corollary}
\label{co:locandcoloc}
The map sending a localising subcategory $\sfC$ of $\StMod G$ to $\sfC^\perp$ induces a bijection
\[ 
\left\{
  \begin{gathered} \text{tensor closed localising}\\
\text{subcategories of $\StMod G$}
\end{gathered}\, \right\} \;\stackrel{\sim}\lra\; \left\{
\begin{gathered} \text{Hom closed colocalising}\\ \text{subcategories
of $\StMod G$}
\end{gathered}\,  \right \}\, .
\]
The inverse map sends a colocalising subcategory $\sfC$ to $^\perp\sfC$.\qed
\end{corollary}

Predecessors of these results are the analogues for the 
derived category of a commutative noetherian ring by 
Neeman \cite{Neeman:2011a}, and the stable module 
category of a finite group~\cite{Benson/Iyengar/Krause:2012b}. 
Any finite group gives rise to a finite group scheme, and 
we obtain an entirely new proof in that case.

Products of modules tend to be more complicated than coproducts. This
is reflected by the fact that the classification of colocalising
subcategories formally implies the classification of localising
subcategories in terms of $\pi$-supports of $G$-modules; see
\cite[Theorem~9.7]{Benson/Iyengar/Krause:2012b}. So
Theorem~\ref{th:costratification} implies the classification result in
our work presented
in~\cite{Benson/Iyengar/Krause/Pevtsova:2015b}. However, the arguments
in the present work rely heavily on the tools developed
in~\cite{Benson/Iyengar/Krause/Pevtsova:2015b}, which, in turn, depend
on the fundamental results and geometric techniques for the
representation theory and cohomology of finite group schemes from
\cite{Suslin:2006a,Bendel/Friedlander/Suslin:1997b}.

An essential ingredient in the proof of
Theorem~\ref{th:costratification} is a family of $G$-modules, one
arising from each $\pi$-point of $G$.  We call them \emph{point
  modules} and write $\chr_G(\alpha)$, where
$\alpha\colon K[t]/(t^p)\to KG$ is the corresponding $\pi$-point.
They appear already in
\cite[Section~9]{Benson/Iyengar/Krause/Pevtsova:2015b} and play the
role of residue fields in commutative algebra. Indeed, while they are
not usually finite dimensional, they are always endofinite in the
sense of Crawley-Boevey~\cite{Crawley-Boevey:1991a}, as is proved in Proposition~\ref{pr:dual}.  It follows from
results in \cite{Benson/Iyengar/Krause/Pevtsova:2015b} that the
$\pi$-support of $\chr_{G}(\alpha)$ is equal to the prime ideal $\fp$
corresponding to $\alpha$, and that the localising subcategory
generated by $\chr_{G}(\alpha)$ is $\gam_{\fp}\StMod G$, the full
subcategory of $\fp$-local and $\fp$-torsion objects.

As Theorem~\ref{th:pg} we prove that $\chr_{G}(\alpha)$ also 
\emph{cogenerates} $\lam^\fp \StMod G$, the full subcategory 
of $\fp$-local and $\fp$-complete $G$-modules,  in the sense 
of \cite{Benson/Iyengar/Krause:2012b}. This result is an important 
step in the proof of Theorem~\ref{th:costratification}, because the 
subcategories $\lam^\fp \StMod G$, as $\fp$ varies over $\Proj H^*(G,k)$, 
cogenerate $\StMod G$. From this it follows that \emph{every} 
Hom closed colocalising subcategories of $\StMod G$ is cogenerated 
by point modules, which again highlights the special role played by them. 

There is a parallel between point modules and standard objects 
in highest weight categories studied by Cline, Parshall, and Scott~\cite{Cline/Parshall/Scott:1988a}. 
This is explained towards the end of this article. The notation 
$\chr_{G}(\alpha)$ reflects this connection.

\section{Recollections}
In this section we recall  basic notions and results on modules 
over finite group schemes required in this work. Our standard 
references are the books of Jantzen~\cite{Jantzen:2003a} and 
Waterhouse~\cite{Waterhouse:1979a}. For the later parts, and 
for the notation, we follow~\cite{Benson/Iyengar/Krause/Pevtsova:2015b}.

Let $G$ be a finite group scheme over a field $k$. Thus $G$ is an 
affine group scheme such that its coordinate algebra $k[G]$ is finite
dimensional as a $k$-vector space. The $k$-linear dual of $k[G]$ is 
a cocommutative Hopf algebra, called the group algebra of $G$, and
denoted $kG$. We identify $G$-modules with modules over the 
group algebra $kG$. The category of all (left) $G$-modules is denoted $\Mod G$.  

The stable module category $\StMod G$ is obtained from $\Mod G$ 
by identifying two morphisms between $G$-modules when they 
factor through a projective $G$-module. The tensor product of 
$G$-modules passes to $\StMod G$ and we obtain a compactly 
generated tensor triangulated category with suspension $\Omega^{-1}$,
the inverse of the syzygy functor. We use the notation $\sHom_G(M,N)$ 
for the $\Hom$-sets in $\StMod G$.  For details, readers might 
consult Carlson \cite[\S5]{Carlson:1996a} and Happel~\cite[Chapter 1]{Happel:1988a}.
 
In the context of finite groups there is a duality theorem due to
Tate~\cite[Chapter XII, Theorem 6.4]{Cartan/Eilenberg:1956a} that is
helpful in computing morphisms in the stable category. In the proof of
Lemma~\ref{le:pg} we need an extension of this to finite group
schemes which is recalled below.
 
\subsection*{Duality} 
Given a $k$-vector space $V$, we set $V^\vee := \Hom_k(V,k)$ to be the
dual vector space. If $V$ is a $G$-module, then $V^\vee$ can also be
endowed with a structure of a $G$-module using the Hopf algebra
structure of $kG$.

Let $G^\op$ denote the opposite group scheme that is given by the
group algebra $(kG)^\op$.  Given a $G^\op$-module $M$, we write
$DM:= \Hom_k(M,k)$ for the dual vector space considered as a
$G$-module. Let
\[
\tau = D \circ \Tr \colon \stmod G\xra{\ \sim\ } \stmod G\,
\] 
be the composition of the duality functor $D$ and the transpose
$\Tr\colon \stmod G \to \stmod G^\op$; see
\cite[III.4]{Skowronski/Yamagata:2011a} for the definition.  For any
$G$-module $M$ and finite dimensional $G$-module $N$, there is a
natural isomorphism of vector spaces
\begin{equation}
\label{eq:tate}
\sHom_G(N, M)^\vee \cong \sHom_G(M, \Omega^{-1} \tau N)\,.
\end{equation} 
This isomorphism can be deduced from a formula of Auslander and 
Reiten~\cite[Proposition~I.3.4]{Auslander:1978a}---see 
also \cite[Corollary p.~269]{Krause:2003a}--- which yields the first isomorphism below 
\[
 \sHom_G(N, M)^\vee \cong \Ext^{1}_{G}(M,\tau N)\cong \sHom_{G}(M,\Omega^{-1} \tau N)\,.
\]
When $kG$ is symmetric (in particular, whenever $G$ is a finite
group), we have $\tau N \cong \Omega^2 N $. This follows from
\cite[IV.8]{Skowronski/Yamagata:2011a} and reduces
\eqref{eq:tate} to  Tate duality. 

\subsection*{Extending the base field} 
Let $G$ be a finite group scheme over a field $k$. If $K$ is 
a field extension of $k$, we write $K[G]$ for $K \otimes_k k[G]$, 
which is a commutative Hopf algebra over $K$. This defines a 
finite group scheme over $K$ denoted $G_K$.  We have a natural 
isomorphism $KG_K\cong K\otimes_k kG$ and we simplify 
notation by writing $KG$. The restriction functor
\[
\res^K_k\colon \Mod G_K \lra \Mod G
\]
admits a left adjoint that sends a $G$-module $M$ to
\[ 
M_K:=K \otimes_k M,
\]
and a right adjoint sending $M$ to
\[ 
M^K := \Hom_k(K,M).
\] 
The next result tracks how these functors interact with taking 
tensors and modules of homomorphisms. We give proofs for 
lack of an adequate reference.

\begin{lemma}
\label{le:restriction}
  Let $K$ be a field extension of $k$.  For a $G_K$-module $M$ and a
  $G$-module $N$, there are natural isomorphisms of $G$-modules:
\begin{align*}
\res^K_k(M\otimes_K N_K)&\cong (\res^K_k M)\otimes_k N,\\
\res^K_k\Hom_K(M,N^K)&\cong\Hom_k(\res^K_k M,N).
\end{align*}
\end{lemma}
\begin{proof}
The first isomorphism is clear since the $k$-linear isomorphism
\[M\otimes_K (K\otimes_k N)\cong (M\otimes_K K)\otimes_k N\cong
M\otimes_k N\]
is compatible with the diagonal $G$-actions. 

The second isomorphism follows from the first one, because the functor
\begin{align*}
\res^K_k\Hom_K(M,(-)^{K}) &\quad\text{is right adjoint to}\quad
\res^K_k(M\otimes_K (-)_{K}),
\intertext{while the functor} 
\Hom_k(\res^K_kM,-) &\quad\text{is right
adjoint to}\quad (\res^K_kM)\otimes_k-. \qedhere
\end{align*} 
\end{proof}

\subsection*{Subgroup schemes}

For each subgroup scheme $H$ of $G$ restriction is a functor
\[ 
\res^G_H\col \Mod G \lra \Mod H.
\] 
This has a right adjoint called \emph{induction}
\[ 
\ind^G_H\col \Mod H \lra \Mod G
\] 
as described in \cite[I.3.3]{Jantzen:2003a}, and a left adjoint called \emph{coinduction}
\[ 
\coind^G_H\col \Mod H \lra \Mod G
\] 
as described in \cite[I.8.14]{Jantzen:2003a}.

\begin{lemma}
\label{le:subgroup}
Let $H$ be a subgroup scheme of $G$. For any $H$-module $M$ and $G$-module $N$ there are natural isomorphisms:
\begin{align*}
\coind^G_H(M\otimes_k \res^G_HN)&\cong (\coind^G_H M)\otimes_k N\\
\ind^G_H\Hom_k(M,\res^G_H N) &\cong\Hom_k(\coind^G_H M,N). 
\intertext{In particular, for $M=k$ these give isomorphisms:}
\coind^G_H\res^G_HN&\cong (\coind^G_H k)\otimes_k N\\
\ind^G_H\res^G_H N&\cong\Hom_k(\coind^G_H k,N).
\end{align*}
\end{lemma}

\begin{proof}  
Recalling that $\coind^{G}_{H}= kG\otimes_{kH} - $,  the first isomorphism follows from associativity of tensor products:
\begin{align*}
\coind^G_H(M\otimes_k \res^G_H N)&\cong  kG\otimes_{kH} (M\otimes_k \res^G_H N)\\
&\cong( kG\otimes_{kH} M)\otimes_k N\\
&\cong (\coind^G_H M)\otimes_k N.
\end{align*}

The second isomorphism follows from the first one, because the functor
\begin{align*}
\ind^G_H\Hom_k(M,-)\res^G_H&\quad\text{is right adjoint to}\quad
\coind^G_H(M\otimes_k-)\res^G_H,
\intertext{while the functor} 
\Hom_k(\coind^G_HM,-) &\quad\text{is right
adjoint to}\quad (\coind^G_HM)\otimes_k-.\qedhere
\end{align*} 
\end{proof}

\subsection*{Cohomology and $\pi$-points}
Let $k$ be a field of positive characteristic $p$ and $G$ a finite group scheme over $k$. We write $H^*(G,k)$ for the cohomology algebra of $G$ and $\Proj H^*(G,k)$ for the set of its homogeneous prime ideals not containing $H^{\ges 1}(G,k)$, the elements of positive degree.

A \emph{$\pi$-point} of $G$, defined over a field extension $K$ of $k$, is a morphism of $K$-algebras
\[ 
\alpha\colon K[t]/(t^p) \lra KG
\] 
that factors through the group algebra of a unipotent abelian subgroup scheme of $G_{K}$, and such that $KG$ is flat when viewed as
a left (equivalently, as a right) module over $K[t]/(t^{p})$ via $\alpha$. Given such an $\alpha$, restriction yields a functor
\[ 
\alpha^{*}\colon \Mod KG \lra \Mod K[t]/(t^{p})\,.
\]
We write $H^{*}(\alpha)$ for the composition of homomorphisms of $k$-algebras
\[ 
H^{*}(G,k) = \Ext^{*}_{G}(k,k) \xra{\ K\otimes_{k}-\ } \Ext^{*}_{G_K}(K,K) \xra{\ \alpha^*\ } \Ext^{*}_{K[t]/(t^{p})}(K,K).
\] 
The radical of the ideal $\Ker H^{*}(\alpha)$ is a prime ideal in $H^*(G,k)$, and the assignment $\alpha\mapsto \sqrt{\Ker H^{*}(\alpha)}$ yields a bijection between the equivalence classes of $\pi$-points and $\Proj H^*(G,k)$; see \cite[Theorem~3.6]{Friedlander/Pevtsova:2007a}.  Recall that $\pi$-points $\alpha\colon K[t]/(t^{p})\to KG$ and $\beta\colon L[t]/(t^{p})\to LG$ are \emph{equivalent} if for every
$G$-module $M$ the module $\alpha^*(M_K)$ is projective if and only if $\beta^*(M_L)$ is projective.  In the sequel, we identify a prime in
$\Proj H^{*}(G,k)$ and the corresponding equivalence class of $\pi$-points.

Given a point in $\Proj H^{*}(G,k)$, there is some flexibility in choosing a $\pi$-point representing it. This will be important in the sequel.

\begin{remark}
\label{re:choosing-pi}
We call a group scheme $\mcE$ quasi-elementary if there is an isomorphism $\mcE \cong \mathbb G_{a(r)} \times E$ where $\mathbb G_{a(r)}$ is the $r^{\rm th}$ Frobenius kernel of the additive group $\mathbb G_a$ and $E$ is an elementary abelian $p$-group. 

By \cite[Proposition~4.2]{Friedlander/Pevtsova:2005a}, given a $\pi$-point $\alpha\colon K[t]/(t^{p})\to KG$, there exists an equivalent  $\pi$-point $\beta\colon K[t]/(t^{p})\to KG$ that factors through a quasi-elementary subgroup scheme of $G_{K}$. 

A point $\fp$ in $\Proj H^*(G,k)$ is \emph{closed} if there is no point in $\Proj H^*(G,k)$ properly containing it as a prime ideal. Then there exists a $\pi$-point $\alpha\colon K[t]/(t^{p})\to KG$ such that $K$ is finite dimensional over $k$; see \cite[Theorem~4.2]{Friedlander/Pevtsova:2007a}. In view of the preceding paragraph, one may choose an $\alpha$ that factors through a quasi-elementary subgroup scheme of $G_{K}$.
\end{remark}

\subsection*{Local cohomology and completions}
We recall from \cite{Benson/Iyengar/Krause:2008a, Benson/Iyengar/Krause:2012b} the definition of local cohomology and completion for
$G$-modules.

The algebra $H^*(G,k)$ acts on $\StMod G$. This means that for $G$-modules $M,N$ there is a natural action of $H^*(G,k)$ on
\[
\sHom^*_G(M,N)=\bigoplus_{i\in\mathbb Z}\sHom_G(\Omega^i M,N)
\]
via the homomorphism of $k$-algebras
\[
-\otimes_k M\colon H^*(G,k)=\Ext_G^*(k,k)\lra\sHom^*_G(M,M).
\]

Fix $\fp\in\Proj H^*(G,k)$. There is a localisation functor
$\StMod G\to \StMod G$ sending $M$ to $M_\fp$ such that the natural
morphism $M\to M_\fp$ induces an isomorphism
\[
\sHom^*_G(-,M)_\fp\xra{\ \sim\ }\sHom^*_G(-,M_\fp)
\]
when restricted to finite dimensional $G$-modules. A $G$-module $M$ is
called \emph{$\fp$-local} if $M\xra{\sim} M_\fp$ and we write
$(\StMod G)_\fp$ for the full subcategory of $\fp$-local $G$-modules.
The module $M$ is \emph{$\fp$-torsion} if $M_\fq=0$ for all
$\fq\in\Spec H^*(G,k)$ that do not contain $\fp$. There is a
colocalisation functor $\gam_{\mcV(\fp)}\colon \StMod G\to \StMod G$
such that the natural morphism $\gam_{\mcV(\fp)}(M)\to M$ is an
isomorphism if and only $M$ is $\fp$-torsion. The functor
$\gam_{\mcV(\fp)}$ admits a right adjoint, denoted $\lam^{\mcV(\fp)}$
and called \emph{$\fp$-completion}. We say that $M$ is
\emph{$\fp$-complete} if the natural map $M\to \lam^{\mcV(\fp)}M$ is
an isomorphism.

The functor $\gam_\fp\colon \StMod G\to \StMod G$ sending 
$M$ to $\gam_{\mcV(\fp)} (M_\fp)$ gives \emph{local cohomology at $\fp$}. 
It has a right adjoint $\lam^\fp\colon\StMod G\to\StMod G$ 
that plays the role of \emph{completion at $\fp$}, for modules 
over commutative rings.

\subsection*{Koszul objects and reduction to closed points} 
\label{sec:koszul}
For a cohomology class $\zeta$ in $H^*(G,k)$, let $\kos k{\zeta}$ be  a mapping cone of the morphism $k \to \Omega^{-d}k$ 
in $\StMod G$ defined by $\zeta$.  Note that $\kos k{\zeta} \cong \Omega^{-d-1}L_\zeta$  where $L_\zeta$ is the Carlson module~\cite{Carlson:1983} defined by $\zeta$. For a homogeneous ideal $\fa$ in $H^*(G,k)$, we pick a system of homogeneous generators $\zeta_1, \ldots, \zeta_n$, and define a Koszul object  $\kos k{\fa}$ to be 
\[
\kos k{\fa} := \kos k{\zeta_1} \otimes_k \ldots \otimes_k \kos k{\zeta_n}.
\]
Observe that the map $k\to \Omega^{-d}k$ defined by $\zeta$ becomes an isomorphism when localised at any prime ideal $\fp$ of $H^{*}(G,k)$ not containing $\zeta$. Given this the next result is \cite[Theorem~8.8]{Benson/Iyengar/Krause/Pevtsova:2015b}.

\begin{theorem}
\label{thm:reduction}
Let $\fp$ be a point in $\Proj H^*(G,k)$. There exists a field extension $L/k$ and an ideal $\fq$ of $H^*(G_L,L)$ with radical $\sqrt \fq$ a closed point in  $\Proj H^*(G_L,L)$ lying over $\fp$ such that there is an isomorphism   
\[ 
\res^L_k (\kos L{\fq}) \cong (\kos k{\fp})_\fp.
\]
\end{theorem} 

The construction of $\kos L{\fq}$ involves a choice of  generators for $\fq$, so the theorem effectively states that there exist 
an ideal $\fq$ \textit{and} a choice of generators that produces the Koszul object with required properties. For details, see \cite[Section 8]{Benson/Iyengar/Krause/Pevtsova:2015b}.

\subsection*{Brown representability}  
\label{sec:brown}
Let $C$ be a finite dimensional $G$-module and $I$ an injective $H^*(G,k)$-module.  Recall that $H^*(G,k)$ acts on $\sHom^*_G(C,M)$ for any $M \in \StMod G$ and consider the contravariant functor
\[
\Hom_{H^*(G,k)}(\sHom^*_G(C,-),I)\colon \StMod G \lra \mathsf{Ab} 
\] 
This functor takes triangles to exact sequences and coproducts to products. Hence, 
by the contravariant version of Brown Representability (see \cite{Brown:1965} or \cite{Neeman:1996}), 
there exists a $G$-module $T_{C}(I)$ such that 
 \begin{equation}
 \label{eq:brown}
 \Hom_{H^*(G,k)}(\sHom^*_G(C,-),I) \cong \sHom_G(-,T_{C}(I)).
\end{equation}
We refer to \cite{Benson/Iyengar/Krause:2012b, Benson/Krause:2002a}
for details about these modules.

\subsection*{Support and cosupport}
The following definitions of $\pi$-support and $\pi$-cosupport of a $G$-module $M$ are from \cite{Friedlander/Pevtsova:2007a} and
\cite{Benson/Iyengar/Krause/Pevtsova:2015a} respectively. We set
\begin{align*}
\pisupp_{G}(M) &:= \{\fp\in\Proj H^*(G,k) \mid \text{$\alpha_\fp^*(M_K)$ is not projective}\}, \\
\picosupp_{G}(M) &:= \{\fp\in\Proj H^*(G,k) \mid \text{$\alpha_\fp^*(M^K)$ is not projective}\}.
\end{align*}
Here $\alpha_\fp\colon K[t]/(t^p)\to KG$ denotes a $\pi$-point corresponding to $\fp$. Both $\pisupp$ and $\picosupp$ are well
defined on the equivalence classes of $\pi$-points \cite[Theorem~2.1]{Benson/Iyengar/Krause/Pevtsova:2015a}.

The local cohomology functors $\gam_\fp$ and their right adjoints $\lam^\fp$ yield alternative notions of support and cosupport for a
$G$-module $M$; see \cite{Benson/Iyengar/Krause:2008a,  Benson/Iyengar/Krause:2012b}. We set
\begin{align*}
\supp_{G}(M) &:=\{\fp \in \Proj H^{*}(G,k)\mid \gam_{\fp}M \neq 0\},\\
\cosupp_{G}(M)& := \{\fp\in\Proj H^*(G,k) \mid \lam^\fp M\neq 0\}.
\end{align*}
It is an important fact that these notions agree with the ones defined via $\pi$-points. This has been proved in \cite{Benson/Iyengar/Krause/Pevtsova:2015b} and will be used freely throughout this work.

\begin{theorem}
\label{th:bik=fp}
For every $G$-module $M$ there are equalities
\[
\pisupp_{G}(M)=\supp_{G}(M) \quad\text{and}\quad \picosupp_{G}(M)=\cosupp_{G}(M).
\]
\end{theorem}

\begin{proof}
See Theorems~6.1 and 9.3 in \cite{Benson/Iyengar/Krause/Pevtsova:2015b}.
\end{proof}

For ease of reference we recall basic facts concerning support and cosupport.

\begin{remark}
\label{re:properties}
Let $M$ and $N$ be $G$-modules.
\begin{enumerate}[\quad\rm(1)]
\item $M$ is projective if and only if $\supp_{G}(M)=\varnothing$, if and only if $\cosupp_{G}(M)=\varnothing$.
\item $\supp_{G}(M)$ and $\cosupp_{G}(M)$ have the same maximal elements with respect to inclusion.
\item $\supp_{G}(M\otimes_{k}N)  = \supp_{G}(M)\cap \supp_{G}(N)$.
\item $\cosupp_{G}\Hom_{k}(M,N)  = \supp_{G}(M)\cap \cosupp_{G}(N)$.
\item $\supp_{G}(k) = \Proj H^{*}(G,k) = \cosupp_{G}(k)$.
\end{enumerate}

Keeping in mind Theorem~\ref{th:bik=fp}, parts (1) and (2) are recombinations of \cite[Theorem~5.3 and Corollary~9.4]{Benson/Iyengar/Krause/Pevtsova:2015b}. Parts (3) and (4) are from \cite[Theorem~3.4]{Benson/Iyengar/Krause/Pevtsova:2015a}, while (5) is contained in  \cite[Lemma~3.5]{Benson/Iyengar/Krause/Pevtsova:2015a}.
\end{remark}

\begin{remark}
\label{re:koszul}  
For an ideal $\fa$ in $H^{*}(G,k)$ we write $\mcV(\fa)$  for the closed subset  of those points in $\Proj H^*(G,k)$ corresponding to homogeneous prime ideals containing $\fa$. 

Let $\zeta_1, \ldots, \zeta_n$ be a system of homogeneous generators of an ideal $\fa \subset H^*(G,k)$. By a theorem of Carlson\cite{Carlson:1983}, one has $\supp_G(\kos k{\zeta}) = \mcV(\zeta)$ for any $\zeta \in H^d(G,k)$. The tensor product property, recalled in Remark~\ref{re:properties}, now implies that 
\[
\supp_G(\kos k{\fa}) = \mcV(\zeta_1) \cap \ldots \cap  \mcV(\zeta_n) = \mcV(\fa). 
\]
In particular, for $L$ and $\fq$ as in Theorem~\ref{thm:reduction}, one gets
\[
\supp_{G_L}(\kos L{\fq}) = \mcV(\fq) = \{\sqrt \fq\} \subset \Proj H^*(G_L,L),
\]
since $\sqrt \fq$ is a closed point in $\Proj H^*(G_L,L)$. 
\end{remark}

\section{Point modules}

In this section we discuss a distinguished class of $G$-modules that
correspond to a $\pi$-point. Later on we will see that these modules
serve as cogenerators of colocalising subcategories.

\subsection*{Point modules} Fix a $\pi$-point $\alpha\colon k[t]/(t^p)
\to kG$. The restriction functor
\[ \alpha^{*}\colon \Mod G \lra \Mod k[t]/(t^{p})
\] admits a left adjoint and a right adjoint:
\[ \alpha_*:=kG\otimes_{k[t]/(t^{p})} -\qquad\text{and}
\qquad\alpha_!:=\Hom_{k[t]/(t^{p})}(kG,-)\,.
\] These functors are isomorphic, as the next result asserts. 

\begin{theorem}
\label{th:gorenstein} For any $\pi$-point $\alpha\colon k[t]/(t^p) \to
kG$ and $k[t]/(t^p)$-module $M$, there is a natural isomorphism of
$G$-modules:
\[ \alpha_*(M) \cong \alpha_{!}(M).
\] 
\end{theorem}

\begin{proof} 
  It is convenient to set $R:=k[t]/(t^{p})$. It is easy to verify that
  the $R$-module $\Hom_{k}(R,k)$ is isomorphic to $R$. This will be
  used further below. We will also use the fact that $kG$ is a
  Frobenius algebra, that is to say that there is an isomorphism of
   $G$-modules
\[
kG\cong \Hom_{k}(kG,k).
\]
See \cite[Lemma~I.8.7]{Jantzen:2003a}, and also \cite[Chapter VI,
Theorem~3.6]{Skowronski/Yamagata:2011a}. This justifies the third step
in the following chain of isomorphisms of  $G$-modules:
\begin{equation}
\label{eq:frobenius}
\Hom_{R}(kG,R)\cong \Hom_{R}(kG,\Hom_{k}(R,k))\cong \Hom_{k}(kG,k)\cong kG.
\end{equation}
The second is standard adjunction.

We are now ready to justify the stated result.  Consider first the
case when $G$ is abelian. Then $kG$ and $\Hom_{R}(kG,R)$ also have
$G^\op$-actions. As $G$ is abelian, the
isomorphism~\eqref{eq:frobenius} is compatible with these
structures. It follows that it is also compatible with the induced
$R^\op$-actions on $kG$ and $\Hom_{R}(kG,R)$. This justifies the
second isomorphism below
\[
\alpha_{!}(M) = \Hom_{R}(kG,M) \cong \Hom_{R}(kG,R)\otimes_{R}M \cong kG\otimes_{R} M = \alpha_{*}(M).
\]
The first isomorphism holds because $kG$ is a finitely generated projective $R$-module. The composition of the maps is the desired isomorphism.
 
Let now $G$ be an arbitrary finite group scheme. By definition, the
$\pi$-point $\alpha$ factors as $R \xra{\beta} kU \hookrightarrow kG$, where
$U$ is an unipotent abelian subgroup scheme of $G$.  Note that
$\beta_{*}=\beta_{!}$ by what we have already verified, since $U$ is
abelian. Observing that $\alpha_*=\coind^G_U\beta_*$ and
$\alpha_!=\ind^G_U\beta_!$, it thus remains to show that
$\coind_U^G\cong\ind_U^G$. By \cite[I.8.17]{Jantzen:2003a}, there is
an isomorphism
\[ 
\coind_U^G (M) \cong \ind_U^G (M \otimes_k (\delta_G){\da_U}\delta_U^{-1})
\] 
where $\delta_G$ and $\delta_U$ are certain characters of $G$ and $U$,
respectively. Since $U$ is a unipotent group scheme, it has no
nontrivial characters; see, for example,
\cite[8.3]{Waterhouse:1979a}. This yields the last claim and therefore
the proof is complete.
\end{proof}

\begin{definition} 
Let $K$ be a field extension of $k$ and $\alpha\colon K[t]/(t^p)\to KG$ a $\pi$-point. We call the $G$-module
\[ 
\chr_G(\alpha):=\res^{K}_{k}\alpha_*(K)\cong \res^{K}_{k}\alpha_!(K)
\] 
the \emph{point module} corresponding to $\alpha$.
\end{definition}

As an example, we describe the point modules for the Klein four group, following the description of the $\pi$-points in \cite[Example~2.3]{Friedlander/Pevtsova:2007a}; see also \cite[Example~2.6]{Benson/Iyengar/Krause/Pevtsova:2015a}.

\begin{example}
\label{ex:klein} 
Let $V=\bbZ/2 \times \bbZ/2$ and $k$ a field of characteristic two. The group algebra $kV$ is isomorphic to $k[x,y]/(x^{2},y^{2})$,
where $x+1$ and $y+1$ correspond to the generators of $V$, and $\Proj H^{*}(V,k)\cong \bbP^{1}_{k}$. A $kV$-module $M$ is given by a
$k$-vector space together with two $k$-linear endomorphisms $x_M$ and $y_M$ representing the action of $x$ and $y$ respectively.

For each closed point $\fp\in \bbP^1_k$ there is some finite field extension $K$ of $k$ such that $\bbP^1_K$ contains a rational point
$[a,b]$ over $\fp$ (using homogeneous coordinates). The $\pi$-point corresponding to $\fp$ is represented by the map of $K$-algebras
\[ 
K[t]/(t^{p})\lra K[x,y]/(x^{2},y^{2})\quad\text{where $t\mapsto ax + by$,}
\]
and the corresponding point module is given by $\Delta=K\oplus K$ together with
\[
x_\Delta=\left[\begin{matrix}0&0\\ b&0\end{matrix}\right]
\quad\text{and}\quad
y_\Delta=\left[\begin{matrix}0&0\\ a&0\end{matrix}\right].
\]
Now let $K$ denote the field of rational functions in a variable
$s$. The generic point of $\bbP^{1}_{k}$ then corresponds to the map
of $K$-algebras
\[ 
K[t]/(t^{p})\lra K[x,y]/(x^{2},y^{2}) \quad\text{where $t\mapsto x+sy$,}
\]
and the corresponding point module is given by $\Delta=K\oplus K$ together with
\[
x_\Delta=\left[\begin{matrix}0&0\\ s&0\end{matrix}\right]
\quad\text{and}\quad
y_\Delta=\left[\begin{matrix}0&0\\ 1&0\end{matrix}\right].
\]
\end{example}

The next example illustrates that the $G$-module $\chr_{G}(\alpha)$ depends on $\alpha$ and not only on
the point in $\Proj H^{*}(G,k)$ that it represents.

\begin{example} 
Let $k$ be a field of characteristic $p\ge 3$ and set $G:=\bbZ/p \times \bbZ/p$. Thus, $kG= k[x,y]/(x^p, y^p)$ and
$\Proj H^{*}(G,k)=\bbP_k^{1}$. The homomorphism
\begin{align*} 
&\alpha_\lambda\colon k[t]/(t^{p})\lra kG \quad \text{where }
t\mapsto x - \lambda y^{2}
\end{align*} 
defines a $\pi$-point for any $\lambda \in k$, corresponding to the same point in $\bbP_k^{1}$, namely $[1,0]$. On the other hand, the point modules
\[ 
\chr_{G}(\alpha_\lambda)\cong k[x,y]/(x-\lambda y^{2},y^{p}).
\]
are pairwise non-isomorphic; for example, their annihilators
differ. They are also indecomposable, because they are cyclic and $kG$
is a local ring.
\end{example}

The next example shows that point modules need not be indecomposable.

\begin{example} 
  Let $k$ be a field of characteristic $3$ and set
  $G:=\Sigma_3\times\bbZ/3$. The $\pi$-point
  $\alpha\colon k[t]/(t^{3})\to kG$ given by the inclusion
  $\bbZ/3\hookrightarrow G$ as a direct factor yields a point module
  $\Delta_G(\alpha)$ that decomposes into two non-isomorphic
  indecomposable $G$-modules, because it is isomorphic to $k\Sigma_3$.
\end{example}

\subsection*{Endofinite modules}
Let $G$ be a group scheme defined over $k$. A point module defined over a field extension $K$ is finite dimensional, as a $G$-module, if and only if $K$ is finite dimensional over $k$.  Nonetheless, point modules always enjoy a strong finiteness property because they arise as
restrictions of finite dimensional modules.

Let $A$ be any ring. Following Crawley-Boevey \cite{Crawley-Boevey:1991a, Crawley-Boevey:1992a}, we say that an $A$-module $M$
is \emph{endofinite} if it has finite composition length when viewed as a module over its endomorphism ring $\End_A(M)$. The following
result, due to Crawley-Boevey, collects some of the basic properties of endofinite modules. The proof employs the fact that endofinite modules are $\Sigma$-pure-injective.

\begin{theorem}
\label{th:endofinite}
An indecomposable endofinite module has a local endomorphism ring  and any endofinite module can be written essentially uniquely as a
direct sum of indecomposable endofinite modules. Conversely, a direct sum of endofinite modules is endofinite if and only if there are only finitely many isomorphism classes of indecomposables involved. 

  The class of endofinite modules is closed under finite direct sums,
  direct summands, and arbitrary products or direct sum of copies of
  one module.
\end{theorem}

\begin{proof}
See Section~1.1 in  \cite{Crawley-Boevey:1991a} and Section~4 in \cite{Crawley-Boevey:1992a}.
\end{proof}

For an $A$-module $M$, we write $\Add(M)$ for the full subcategory of $A$-modules that are direct summands of direct sums of copies of
$M$. Analogously, $\Prod(M)$ denotes the subcategory of all direct summands of products of copies of $M$. For an endofinite module $M$ it follows from Theorem~\ref{th:endofinite} that $\Add(M)$ and $\Prod(M)$ coincide: they consist of all direct sums of indecomposable direct
summands of $M$.  This observation explains the formal part of the following proposition.

\begin{proposition}
\label{pr:dual}
For any $\pi$-point $\alpha$ of $G$,  the $G$-module $\chr_{G}(\alpha)$ is endofinite and there is an equality
\[
\Add(\chr_G(\alpha)) =\Prod(\chr_G(\alpha))\,.
\]
\end{proposition}

\begin{proof}
Let $\alpha\colon K[t]/(t^p)\to KG$ be the given $\pi$-point.  Then $\chr_G(\alpha)$ is a $kG$-$K$-bimodule and there is a homomorphism of rings $K\to \End_{G}(\chr_G(\alpha))$. In particular, $\dim_{K}(\chr_{G}(\alpha))$ is an upper bound for the length of $\chr_G(\alpha)$ as a module over $\End_G(\chr_G(\alpha))$.  Since one has inequalities
\[
\dim_{K}(\chr_{G}(\alpha)) = \frac 1p\dim_{K}(KG) \leq \dim_{K}(KG) < \infty\,,
\]
it follows that $\chr_{G}(\alpha)$ is endofinite. The remaining assertion is by Theorem~\ref{th:endofinite}.
\end{proof}

\subsection*{Support and cosupport}
Next we explain how point modules can be used to compute support and cosupport; this is partly why we are interested in them.

\begin{proposition}
\label{pr:cosupp-point} 
Let $\alpha$ be a $\pi$-point corresponding to
$\fp\in\Proj H^*(G,k)$ and $M$ a $G$-module.  The following statements
are equivalent.
\begin{enumerate}[\quad\rm(1)]
\item $\fp\not\in\cosupp_G(M)$;
\item $\Hom_{k}(\chr_{G}(\alpha),M)$ is projective;
\item $\sHom_{G}(\chr_{G}(\alpha),M)= 0$; 
\item $\sHom_{G}^{*}(\chr_{G}(\alpha),M)= 0$.  \qed
\end{enumerate}
\end{proposition}
\begin{proof} 
  The equivalences (1) $\Leftrightarrow$ (2) $\Leftrightarrow$ (3) are
  \cite[Lemma~9.2]{Benson/Iyengar/Krause/Pevtsova:2015b}.  

(1) $\Leftrightarrow$ (4): With $\alpha$ the map $K[t]/(t^p)\to KG$
adjunctions yield isomorphisms
\[ 
\sHom^*_G(\res^K_k\alpha_*(K),M)\cong \sHom^*_{G_K}(\alpha_*(K),M^K)
\cong \sHom^*_{K[t]/(t^p)}(K,\alpha^{*}(M^K)).
\] 
Clearly, the right hand term vanishes if and only if $\alpha^{*}(M^K)$ is projective.
\end{proof}

Here is the analogous statement for supports. As in the context of
commutative rings, one can use also tensor products with the point
modules to detect support.

\begin{proposition}
\label{pr:supp-point} 
Let $\alpha$ be a $\pi$-point corresponding to
$\fp\in\Proj H^*(G,k)$ and $M$ a $G$-module.  The following statements
are equivalent.
\begin{enumerate}[\quad\rm(1)]
\item $\fp\not\in\supp_G(M)$;
\item $\chr_{G}(\alpha)\otimes_{k}M$ is projective;
\item $\Hom_{k}(M,\chr_{G}(\alpha))$ is projective;
\item $\sHom_G(M,\chr_G(\alpha) )= 0$;
\item $\sHom_G^{*}(M,\chr_G(\alpha) )= 0$.
\end{enumerate}
\end{proposition}

\begin{proof}
(1) $\Leftrightarrow$ (2): Since $\supp_{G}(\chr_{G}(\alpha))=\{\fp\}$, by \cite[Lemma~9.1]{Benson/Iyengar/Krause/Pevtsova:2015b}, Remark~\ref{re:properties}(3) yields the first equivalence below.
\begin{align*}
\fp \not\in\supp_{G}(M) 
	& \iff \supp_{G}(\chr_{G}(\alpha)\otimes_{k}M)=\varnothing \\
	& \iff \chr_{G}(\alpha)\otimes_{k}M \quad \text{is projective}.
\end{align*}
The second one holds because support detects projectivity, by Remark~\ref{re:properties}(1).
 
(1) $\Leftrightarrow$ (4): With $\alpha$ the map $K[t]/(t^p)\to KG$
adjunctions yield isomorphisms
\[ 
\sHom_G(M,\res^K_k\alpha_!(K))\cong \sHom_{G_K}( M_K,\alpha_!(K)) \cong \sHom_{K[t]/(t^p)}(\alpha^{*}(M_K),K).
\] 
Clearly, the right hand term vanishes if and only if $\alpha^{*}(M_K)$ is projective.

(1) $\Leftrightarrow$ (5) is analogous to (1) $\Leftrightarrow$ (4).

(1) $\Rightarrow$ (3): When $\fp$ is not in $\supp_{G}(M)$, it is not
in $\supp_{G}(C\otimes_{k}M)$ for any finite dimensional $G$-module
$C$, by Remark~\ref{re:properties}(3). Thus, the already established
equivalence of conditions (1) and (4) yields that
\[
\sHom_{G}(C,\Hom_{k}(M,\chr_{G}(\alpha)))\cong
\sHom_{G}(C\otimes_{k}M, \chr_{G}(\alpha))=0\,.
\]
Therefore $\Hom_{k}(M,\chr_{G}(\alpha))$ is projective.

(3) $\Rightarrow$ (4): This is clear. 
\end{proof}

In the next result, the claim about the support of $\chr_{G}(\alpha)$ is from \cite[Lemma~9.1]{Benson/Iyengar/Krause/Pevtsova:2015b}, and has been used in the proofs of the Propositions~\ref{pr:cosupp-point} and \ref{pr:supp-point}.

\begin{corollary}
\label{co:point-dual} 
Let $\alpha$ be a $\pi$-point of $G$. A $\pi$-point $\beta$ of $G$ is equivalent to $\alpha$ if and only if $\sHom_{G}^{*}(\chr_{G}(\beta),\chr_{G}(\alpha))\ne 0$.  In particular, there are equalities
\[
\supp_G(\chr_G(\alpha))=\{\fp\}=\cosupp_G(\chr_G(\alpha))
\]
where $\fp$ is the point in $\Proj H^*(G,k)$ corresponding to $\alpha$.
\end{corollary}

\begin{proof} 
  If $\beta$ corresponds to a point $\fq$ in $\Proj H^{*}(G,k)$, then
  $\supp_{G}(\chr_{G}(\beta))=\{\fq\}$, by
  \cite[Lemma~9.1]{Benson/Iyengar/Krause/Pevtsova:2015b}, so
  Proposition~\ref{pr:supp-point} yields that
  $\sHom_{G}^{*}(\chr_{G}(\beta),\chr_{G}(\alpha))$ is non-zero
  precisely when $\fq=\fp$. Given this it follows from
  Proposition~\ref{pr:cosupp-point} that the cosupport of
  $\chr_{G}(\alpha)$ is $\{\fp\}$.
\end{proof}

\section{$\fp$-local and $\fp$-complete objects}

The proof of Theorem~\ref{th:costratification} amounts to showing that for any homogeneous prime ideal $\fp$ of $H^*(G,k)$ the $\fp$-local and $\fp$-complete objects in $\StMod G$ form a minimal Hom closed colocalising subcategory. Here, a Hom closed colocalising
subcategory $\sfC\subseteq\StMod G$ is \emph{minimal} if $\sfC'\subseteq\sfC$ implies $\sfC'=0$ or $\sfC'=\sfC$ for any Hom closed colocalising subcategory $\sfC'\subseteq\StMod G$.

\subsection*{$\fp$-local and $\fp$-complete objects}
We recall from \cite{Benson/Iyengar/Krause:2008a, Benson/Iyengar/Krause:2012b} the definitions and basic facts about
$\fp$-local and $\fp$-complete objects in $\StMod G$.

Fix $\fp\in\Proj H^*(G,k)$.  We write $\gam_\fp\StMod G$ for the full subcategory of $G$-modules $M$ such that
$\gam_\fp(M)\cong M$ and  have from \cite[Corollary~5.9]{Benson/Iyengar/Krause:2008a}
\[
\gam_\fp\StMod G=\{M\in\StMod G\mid\supp_G(M)\subseteq\{\fp\}\}.
\]
From \cite[Corollaries~4.8 and 4.9]{Benson/Iyengar/Krause:2012b}, it follows that a $G$-module $M$ satisfies $\lam^\fp (M)\cong M$ if and only if $M$ is $\fp$-local and $\fp$-complete, and that
\[
\lam^\fp\StMod G=\{M\in\StMod G\mid\cosupp_G(M)\subseteq\{\fp\}\}.
\]
Note that the adjoint pair $(\gam_\fp,\lam^\fp)$ restricts by \cite[Proposition~5.1]{Benson/Iyengar/Krause:2012b} to an equivalence 
\[
\gam_\fp\StMod G\xra{\ \sim\ }\lam^\fp\StMod G.
\]

\subsection*{Cogenerators for $\fp$-local and $\fp$-complete objects}
Given a set $T$ of $G$-modules, let $\Coloc(T)$ denote the smallest colocalising subcategory of $\StMod G$ that contains $T$.  We say that $T$ \emph{cogenerates} a class $\sfC$ of $G$-modules if $\sfC\subseteq \Coloc(T)$.  The class $\sfC$ is \emph{Hom closed} if
for every pair of $G$-modules $M,N$ with $N\in\sfC$, we have $\Hom_k(M,N)\in\sfC$.  We write $\Coloc^{\Hom}(T)$ for the smallest
Hom closed colocalising subcategory that contains $T$.

An object $T$  is a \emph{perfect cogenerator}  of a colocalising subcategory $\sfC\subseteq\StMod G$ if the  the following holds:
\begin{enumerate}[{\quad\rm(1)}]
\item If $M$ is an object in $\sfC$ and $\sHom_G(M,T)=0$ then $M=0$.
\item If a countable family of morphisms $M_i\to N_i$ in
$\sfC$ is such that for all $i$
\[ 
\sHom_G(N_i,T) \lra \sHom_G(M_i,T) 
\] 
is surjective, then so is the induced map
\[ 
\sHom_G(\prod_i N_i,T) \lra \sHom_G(\prod_i M_i,T). 
\]
\end{enumerate}
Any perfect cogenerator is a cogenerator; see \cite[Section~5]{Benson/Iyengar/Krause:2012b}.

Recall from Remark~\ref{re:choosing-pi} that any closed point of $\Proj H^*(G,k)$ is represented by a $\pi$-point $\alpha: K[t]/t^p \to KG$ defined over a finite field extension $K/k$. 

\begin{lemma}
\label{le:pg}  
Let $\alpha\colon K[t]/(t^{p})\to KG$ be a $\pi$-point representing $\fp\in \Proj H^{*}(G,k)$. If $K$ is finite dimensional over $k$, then $\chr_G(\alpha)$  perfectly cogenerates $\lam^\fp\StMod G$.
\end{lemma}

\begin{proof}
We check the conditions (1) and (2) for  $\chr_G(\alpha)$.

(1) If  $M\in \lam^\fp\StMod G$ is non-zero, then $\cosupp_{G}(M)=\{\fp\}$ and hence $\fp$ is in $\supp_{G}(M)$, by Remark~\ref{re:properties}(2). Thus, $\sHom_G(M,\chr_G(\alpha))\ne 0$, by Proposition~\ref{pr:supp-point}.

 (2) Since extension of scalars is left adjoint to restriction of scalars we have
  \[ 
\sHom_G(M,\chr_G(\alpha)) \cong \sHom_{G_K}(M_K,\alpha_*(K)). 
\]
As $\alpha_*(K)$ is finite dimensional as a $G_K$-module, using the duality isomorphism~\eqref{eq:tate} we may rewrite the right hand term as
\[ 
\sHom_{G_K}( \tau^{-1}\Omega(\alpha_*(K)),M_K)^{\vee}.
\] 
So $\sHom_G(N,\chr_G(\alpha)) \to \sHom_G(M,\chr_G(\alpha))$ is
surjective if and only if
\[ 
\sHom_{G_K}(\tau^{-1}\Omega(\alpha_*(K)),M_K) \lra \sHom_{G_K}(\tau^{-1}\Omega (\alpha_*(K)),N_K) 
\]
is injective. It remains to observe that $M\mapsto M_K$ preserves products as $K$ is finite dimensional over $k$. 
\end{proof}

Let $I$ be an injective $H^*(G,k)$-module and $C$ a finite dimensional $G$-module.  
In what follows, we use the representing objects $T_C(I)$
and the Koszul objects $\kos k{\fp}$ defined in \S\ref{sec:koszul}.

\begin{lemma}
\label{le:TI} Fix a point $\fp$ in $\Proj H^{*}(G,k)$ and $I$ an
injective $H^{*}(G,k)$-module.
\begin{enumerate}[\quad\rm(1)]
\item For any finite dimensional $G$-modules $C$, $M$, there is a natural
isomorphism
\[ \Hom_k(M, T_C(I))\cong T_{\Hom_k(M,C)}(I).
\]
\item With $I$ the injective envelope of $H^{*}(G,k)/\fp$, the modules
$\Hom_k(\kos k{\fp}, T_C(I))$, as $C$ varies over the simple
$G$-modules, perfectly cogenerate $\lam^{\fp}\StMod G$.
\end{enumerate}
\end{lemma}

\begin{proof} 
  Recall that $(-)^{\vee}$ denotes the functor $\Hom_{k}(-,k)$. For a
  $G$-module $M$, we consider $M^\vee$ with the diagonal $G$-action,
  and we have
\[\Hom_k(M,-)\cong -\otimes_k  M^\vee\] when $M$ is finite
dimensional. Combining this with standard
adjunctions and the definition of $T_C(I)$ gives the following
isomorphisms, which justify (1).
\begin{align*} 
\sHom_{G}(-,\Hom_k(M, T_C(I))) & \cong
\sHom_{G}(-\otimes_{k} M, T_C(I)) \\ & \cong
\Hom_{H^*(G,k)}(\sHom^*_G(C,-\otimes_{k} M ),I) \\ & \cong
 \Hom_{H^*(G,k)}(\sHom^{*}_{G}(C\otimes_{k}M^{\vee},
-),I) \\ & \cong \Hom_{H^*(G,k)}(\sHom^{*}_{G}(\Hom_{k}(M,C),-),I)\\
&\cong \sHom_{G}(-,T_{\Hom_k(M,C)}(I))
\end{align*}

As to (2), given the isomorphism in (1) applied to $M=\kos k{\fp}$, one can deduce the desired
result by mimicking the proof of
\cite[Proposition~5.4]{Benson/Iyengar/Krause:2012b}.
\end{proof}

For the next result we employ the reduction to closed points technique from \S\ref{sec:koszul}.

\begin{proposition}
\label{pr:res-kos}
Let $\fp$ be a point in $\Proj H^{*}(G,k)$ and $M$ a $\fp$-local
$G$-module. There exists a field extension $L/k$ and an ideal $\fq$ in
$H^{*}(G_{L},L)$ with radical a closed point in $\Proj H^{*}(G_{L},L)$
lying over $\fp$ such that the $G$-module
$\res^L_k \Hom_L(\kos {L}{\fq},M^L)$ is isomorphic to
$\Hom_k(\kos k{\fp},M)$.
\end{proposition}

\begin{proof} By Theorem~\ref{thm:reduction}, we can find  $L$ and $\fq$ such that  there is an isomorphism $\res^L_k({\kos L{\fq}})\cong (\kos k{\fp})_{\fp}$.  Thus there are isomorphisms
\begin{align*}
  \res^L_k \Hom_L(\kos {L}{\fq},M^L)&\cong  \Hom_k(\res^L_k(\kos {L}{\fq}),M)\\
  &\cong \Hom_k((\kos {k}{\fp})_\fp,M)\\
  &\cong \Hom_k(\kos {k}{\fp},M).
\end{align*}
The first one follows from Lemma~\ref{le:restriction} and the last one holds as $M$ is $\fp$-local.
\end{proof}

In what follows, $\Thick(M)$ denotes the thick subcategory of
$\StMod G$ generated by a $G$-module $M$.

\begin{theorem}
\label{th:pg} 
Given $\fp\in\Proj H^*(G,k)$, there exists a $\pi$-point $\alpha\colon K[t]/(t^p)\to KG$ corresponding to $\fp$ 
that factors through a quasi-elementary subgroup scheme of $G_{K}$ and has the following properties:
\begin{enumerate}[\quad\rm(1)] 
\item $\chr_{G}(\alpha)$ is a compact object in $(\StMod G)_{\fp}$.
\item $\Coloc(\chr_G(\alpha)) =\lam^\fp \StMod G$.
\end{enumerate}
\end{theorem}

\begin{proof} 
Let $L$ and $\fq$ be as in Proposition~\ref{pr:res-kos}, and let $\fm = \sqrt \fq$. 
Since $\fm$ is a closed point in $\Proj H^{*}(G_{L},L)$, 
there exists a finite extension $K$ of $L$ and a $\pi$-point
$\alpha\colon K[t]/(t^{p})\to KG$ of $G_{L}$ corresponding 
to $\fm$, and factoring through a quasi-elementary subgroup 
scheme of $G_{K}$; see Remark~\ref{re:choosing-pi}. It then 
follows directly from the definitions that  $\alpha$ corresponds 
to $\fp$, when viewed as a $\pi$-point of $G$.

(1) Set $M:=\kos L{\fq}$. This is a finite dimensional
$G_{L}$-module with support $\{\fm\}$; see Remark~\ref{re:koszul}. 
From the construction it is clear that the $G_{L}$-module $\res^{K}_{L}\alpha_{*}(K)$ 
is also finite dimensional and with support $\{\fm\}$. Thus the
classification~\cite[Corollary~10.2]{Benson/Iyengar/Krause/Pevtsova:2015b}
of tensor closed thick subcategories of $\stmod G_{L}$ yields that
$\res^{K}_{L}\alpha_{*}(K)$ is in $\Thick^{\otimes}(M)$. Any
simple $G_{L}$-module is a direct summand of $S_{L}$, where $S$ is the
sum of representatives of isomorphism classes of simple $G$-modules, so one gets
\[
\res^{K}_{L}\alpha_{*}(K) \in
\Thick(M\otimes_{L}S_{L}).
\]
Applying $\res^{L}_{k}$ and using Lemma~\ref{le:restriction}, one then gets that
\[
\chr_{G}(\alpha)= \res^{L}_{k}\res^{K}_{L}\alpha_{*}(K) \in \Thick((\res^{L}_{k}M)\otimes_{k}S).
\]
It remains to verify that $(\res^{L}_{k}M)\otimes_{k}S$ is a compact
object in $(\StMod G)_{\fp}$. To this end, note that there are
isomorphisms
\[
(\res^{L}_{k}M)\otimes_{k}S \cong (\kos k{\fp})_{\fp} \otimes_{k} S \cong (\kos k{\fp}\otimes_{k}S)_{\fp}
\]
where the first one is by Theorem~\ref{thm:reduction}
and the second is by, for example, \cite[Theorem~8.2]{Benson/Iyengar/Krause:2008a}. 
It remains to note that $\kos k{\fp}\otimes_{k}S$ is a finite dimensional $G$-module and 
hence compact in $\StMod G$, so that its localisation at $\fp$ is compact in $(\StMod G)_{\fp}$.

(2) Let $I$ denote the injective envelope of the $H^*(G,k)$-module $H^*(G,k)/\fp$.  
Since $\supp_{G_{L}}(\kos L{\fq})=\{\fm\}$,  Remark~\ref{re:properties}(4) implies that 
for any finite dimensional $G$-module $C$, the module $\Hom_{L}(\kos{L}{\fq}, T_C(I)^{L})$ 
belongs to $\lam^\fm \StMod G_{L}$. Given the choice of $\alpha$, Lemma~\ref{le:pg} 
thus implies that this module is cogenerated by $\chr_{G_{L}}(\alpha)$. So, by  
Proposition~\ref{pr:res-kos}, the $G$-module $\res^{L}_k\chr_{G_{L}}(\alpha)$, that is 
to say, $\chr_{G}(\alpha)$, cogenerates  $\Hom_k(\kos k{\fp}, T_C(I))$. It remains to 
apply Lemma~\ref{le:TI}(2).
\end{proof}

\subsection*{Minimality}
Next we prove that $\lam^\fp\StMod G$ is a minimal Hom closed colocalising subcategory. This requires further preparation.

\begin{lemma}
\label{le:ind-res}
Let $K$ be a field extension of $k$ and $H$  a subgroup scheme of $G_K$. Set $F=\res^K_k\coind^{G_K}_H(K)$. If $M$ is a $G$-module then
\[ 
\res^K_k\ind^{G_K}_H\res^{G_K}_H(M^K) = \Hom_{k}(F, M). 
\]
When $K$ is a finite extension of $k$, the $G$-module $F$ is finite dimensional over $k$.
\end{lemma}

\begin{proof} 
The desired result is a consequence of the following isomorphisms:
\begin{align*}
\res^K_k\ind^{G_K}_H\res^{G_K}_H(M^K) 
&\cong\res^K_k\Hom_K(\coind^{G_K}_H(K),M^K) \\
&\cong\Hom_k(\res^K_k\coind^{G_K}_H(K),M).
\end{align*}
The first one follows from Lemma~\ref{le:subgroup} and the second from Lemma~\ref{le:restriction}.  The last assertion follows from the fact that, in general, there are inequalities
\[
\dim_{K}\coind^{G_{K}}_{H}(K) = \frac{\dim_{K}(KG)}{\dim_{K}(KH)} \leq \dim_{K}(KG)
\]
and hence the number on the left is finite.
\end{proof}

\begin{lemma}
\label{le:qelem}
Let $\mcE$ be a quasi-elementary group scheme over $K$ and $\beta\colon K[t]/(t^p)\to K\mcE$ a $\pi$-point.
For any  $\mcE$-module $M$, the $\mcE$-module $\beta_!\beta^*(M)$ is in $\Thick(M)$.
\end{lemma}

\begin{proof} 
Note that neither $\beta^{*}$ nor $\beta_{!}$ involve the coproduct on $\mcE$, so we may change that and assume that $KU$ is the group algebra of an elementary abelian $p$-group and that $\beta$ is the inclusion $KH\to K\mcE$, where $H$ is a cyclic subgroup $\mcE$.  
Lemma~\ref{le:ind-res} then yields that $\ind^{\mcE}_{H}\res^{\mcE}_{H}(M)$, that is to say, $\beta_!\beta^*(M)$, equals $\Hom_{k}(F,M)$ for some finite dimensional $\mcE$-module $F$. Since $k$ is the only simple $\mcE$-module, $F$ is in  $\Thick(k)$, and hence $\Hom_{k}(F,M)$ is in $\Thick(\Hom_{k}(k,M))$. It remains to recall that $\Hom_{k}(k,M)\cong M$ as $\mcE$-modules.
\end{proof}

Combining the preceding results one obtains the following.

\begin{proposition}
\label{pr:cg}
Let $\alpha\colon K[t]/(t^p)\to KG$ be a $\pi$-point of $G$ that factors through a quasi-elementary subgroup scheme of $G_{K}$. Then $\res^{K}_{k}\alpha_!\alpha^*(M^K)$ is in $\Coloc^{\Hom}(M)$ for any $G$-module $M$.
\end{proposition}

\begin{proof}
By hypothesis, there exists a quasi-elementary subgroup scheme $U$ of $G_{K}$ such that $\alpha=\gamma\comp\beta$ where  $\beta\colon K[t]/(t^p)\to KU$ and $\gamma\colon KU\to KG$. Then
\[ 
\res^{K}_{k}\alpha_!\alpha^*(M^K)= \res^{K}_{k}\gamma_!\beta_!\beta^*\gamma^*(M^K). 
\]
Since $\beta_!\beta^*\gamma^*(M^K)$ is in $\Thick(\gamma^*(M^K))$ by Lemma~\ref{le:qelem}, one has that
\[ 
\res^{K}_{k}\alpha_!\alpha^*(M^K)\in \Thick(\res^{K}_{k}\gamma_!\gamma^*(M^K)). 
\] 
Since $\res^K_k\gamma_!\gamma^*(M^K)$ is in $\Coloc^{\Hom}(M)$ by Lemma~\ref{le:ind-res}, it follows that
\[
\res^{K}_{k}\alpha_!\alpha^*(M^K)\in\Coloc^{\Hom}(M). \qedhere
\]
\end{proof}

The next result complements Theorem~\ref{th:pg}.

\begin{theorem}
\label{th:cg}
Let $M$ be a $G$-module and $\fp\in\cosupp_{G}(M)$. If $\alpha\colon K[t]/(t^p)\to KG$ is a $\pi$-point that factors through a quasi-elementary subgroup scheme of $G_{K}$ and represents $\fp$, then  $\chr_G(\alpha)$ is in $\Coloc^{\Hom}(M)$.
\end{theorem}

\begin{proof} 
By hypothesis on $\fp$, the $k[t]/(t^{p})$-module $\alpha^*(M^K)$ is not projective, and hence $K$ is in $\Coloc(\alpha^*(M^K))$.  This implies that $\alpha_!(K)$ is in $\Coloc(\alpha_!\alpha^*(M^K))$, and hence, by restriction of scalars, that 
\[ 
\chr_G(\alpha) \text{ is in } \Coloc(\res^{K}_{k}\alpha_!\alpha^*(M^K)). 
\]
Finally, by Proposition~\ref{pr:cg}, the module on the right is in $\Coloc^{\Hom}(M)$.
\end{proof}

\begin{corollary}
\label{co:costratified}
For $\fp\in\Proj H^*(G,k)$, the colocalising subcategory $\lam^\fp\StMod G$ of $\StMod G$ contains no proper non-zero Hom closed colocalising subcategories.
\end{corollary}

\begin{proof}
Fix a $\pi$-point $\alpha$ as in Theorem~\ref{th:pg}, factoring through a quasi-elementary subgroup scheme.  Since $\fp$ is in the $\pi$-cosupport of any non-zero module $M$ in $\lam^\fp\StMod G$, Theorem~\ref{th:cg} yields the inclusion below
\[
\lam^\fp\StMod G = \Coloc(\chr_G(\alpha)) \subseteq \Coloc^{\Hom}(M)\,.
\]
The equality is from Theorem \ref{th:pg}. This is the desired result.
\end{proof}

\begin{proof}[Proof of Theorem~\ref{th:costratification}]
  In the terminology of \cite{Benson/Iyengar/Krause:2012b},
  Corollary~\ref{co:costratified} means that $\StMod G$ is
  \emph{costratified} by the action of $H^*(G,k)$. Given this
  \cite[Corollary~9.2]{Benson/Iyengar/Krause:2012b} yields the desired
  bijection between Hom closed colocalising subcategories of
  $\StMod G$ and subsets of $\Proj H^*(G,k)$.
\end{proof}

\subsection*{Colocalising and localising subcategories}
A key step in the proof of the classification theorem above is that, given a point $\fp$ in $\Proj H^{*}(G,k)$, the point module associated to a certain type of $\pi$-point representing $\fp$ cogenerates $\lam^{\fp}\StMod G$; see~Theorem~\ref{th:pg}. As a corollary of the classification result, it follows that any $\pi$-point may be used, as long as we allow also tensor products with simple modules.

\begin{corollary}
\label{co:point-modules}
For any point $\fp$ in $\Proj H^{*}(G,k)$ and any $\pi$-point representing $\fp$, there is an equality
\[
\Loc^{\otimes}(\chr_{G}(\alpha)) = \gam_{\fp}\StMod G\quad\text{and}\quad
	\Coloc^{\Hom}(\chr_{G}(\alpha)) = \lam^{\fp}\StMod G\,.
\]
\end{corollary}

\begin{proof}
Since $\supp_{G}(\chr_{G}(\alpha))=\{\fp\}$, the first equality is a direct consequence of the bijection between tensor closed localising subcategories of $\StMod G$ and subsets of $\Proj H^{*}(G,k)$ established in \cite[Theorem~8.1]{Benson/Iyengar/Krause/Pevtsova:2015b}. In the same vein, the second equality follows from Theorem~\ref{th:costratification}, since $\cosupp_{G}(\chr_{G}(\alpha))=\{\fp\}$.
\end{proof}

Given a subcategory $\sfC$ of $\StMod G$ we set
\[
\supp_{G}(\sfC):=\bigcup_{M\in \sfC}\supp_{G}(M)\quad\text{and}\quad
\cosupp_{G}(\sfC):=\bigcup_{M\in \sfC}\cosupp_{G}(M)\,.
\]
For any subset $\mcU\subseteq \Proj H^*(G,k)$ set
\[
\cl(\mcU):=\{\fp\in \Proj H^*(G,k)\mid \fp\subseteq\fq\text{ for  some }\fq\in\mcU\}.
\] 
This is the closure of $\mcU$ with respect to the Hochster dual of the Zariski topology \cite{Hochster:1969a}, and we call $\mcU$
\emph{generalisation closed} if $\cl(\mcU)=\mcU$. 
\begin{corollary}
\label{co:localising-colocalising}
For a subcategory $\sfC\subseteq\StMod G$ the following are
equivalent:
\begin{enumerate}[\quad\rm(1)]
\item $\sfC$ is a tensor closed localising subcategory and closed under all products;
\item $\sfC$ is a Hom closed colocalising subcategory and closed under all coproducts.
\end{enumerate}
In that case we have $\supp_G(\sfC)=\cosupp_G(\sfC)$ and this set is
generalisation closed. Moreover, any generalisation closed subset of
$\Proj H^*(G,k)$ arises in that way.
\end{corollary}

\begin{proof}
  In \cite{Benson/Iyengar/Krause/Pevtsova:2015b} it is proved that, as
  a tensor triangulated category, $\StMod G$ is stratified by
  $H^*(G,k)$.  It follows that the assignment $\sfC\mapsto\supp_G(\sfC)$
  yields a bijection between the tensor closed localising subcategories
  of $\StMod G$ that are closed under all products and the
  generalisation closed subsets of $\Proj H^*(G,k)$. This can be
  verified by mimicking the argument used to prove the implication
  (a) $\Leftrightarrow$ (c) of \cite[Theorem~11.8]{Benson/Iyengar/Krause:2011a};
  see also \cite[Theorem~6.3]{Benson/Iyengar/Krause:2011a}. The
  desired assertion now follows from the bijection between localising
  and colocalising subcategories (Corollary~\ref{co:locandcoloc}),
  noticing that for any tensor ideal localising subcategory $\sfC$ we
  have
\[
\supp_G(\sfC)\sqcup\cosupp_G(\sfC^\perp)=\Proj H^*(G,k).\qedhere
\]
\end{proof}

For any generalisation closed subset $\mcU\subseteq\Proj H^*(G,k)$ we set
\begin{align*}
  (\StMod G)_\mcU :=\{M\in\StMod G\mid
  \supp_G(M)\subseteq\mcU\}\,.
\end{align*}
We collect some basic properties of this category.

\begin{remark}
There is an equality 
\[
(\StMod G)_\mcU=\{M\in\StMod G\mid \cosupp_G(M)\subseteq\mcU\}
\] 
and this is compactly generated as a triangulated category. The first assertion is justified by
Remark~\ref{re:properties}(2), and  compact generation follows from
the fact that 
\[(\StMod G)_\mcU=\gam_{\mcU^{\sfc}}(\StMod_{G})^\perp\]
where $\mcU^{\sfc}:=\Proj H^{*}(G,k)\setminus \mcU$.  Indeed, the
subset $\mcU^{\sfc}$ is specialisation closed, so
$\gam_{\mcU^{\sfc}}(\StMod_{G})$ is compactly generated (see, for
example, \cite[Proposition~2.7]{Benson/Iyengar/Krause:2011a}). Now the
assertion is a formal consequence of \cite[Theorem~2.1]{Neeman:1992a} and
\cite[Theorem~9.1.16]{Neeman:2001a}.

Given generalisation closed subsets
$\mcV\subseteq \mcU\subseteq\Proj H^*(G,k)$, it follows from Brown
representability \cite{Neeman:2001a} that the inclusion
\[
(\StMod G)_\mcV\lra (\StMod G)_\mcU
\] 
admits a left adjoint and a right adjoint, because the functor
preserves products and coproducts.
\end{remark}

Now fix a point $\fp$ in $\Proj H^*(G,k)$ and consider the 
generalisation closure of $\fp$. Then $(\StMod G)_{\le\fp}$ equals the
full subcategory of $\fp$-local $G$-modules and we obtain the following
pair of equivalent recollements.
\[\begin{tikzcd}
  (\StMod G)_{<\fp}\arrow[tail]{rr} &&(\StMod G)_{\le\fp}
  \arrow[twoheadrightarrow,yshift=-1.5ex]{ll}
  \arrow[twoheadrightarrow,yshift=1.5ex]{ll}
  \arrow[twoheadrightarrow]{rr}[description]{\gam_\fp} &&\gam_\fp(\StMod G)
  \arrow[tail,yshift=-1.5ex]{ll}{\lam^\fp}
  \arrow[tail,yshift=1.5ex]{ll}[swap]{\mathrm{incl}},\\
  (\StMod G)_{<\fp}\arrow[tail]{rr} &&(\StMod G)_{\le\fp}
  \arrow[twoheadrightarrow,yshift=-1.5ex]{ll}
  \arrow[twoheadrightarrow,yshift=1.5ex]{ll}
  \arrow[twoheadrightarrow]{rr}[description]{\lam^\fp} &&\lam^\fp(\StMod G)
  \arrow[tail,yshift=-1.5ex]{ll}{\mathrm{incl}}
  \arrow[tail,yshift=1.5ex]{ll}[swap]{\gam_\fp}
\end{tikzcd}\]
Note that for a $\pi$-point $\alpha$ representing $\fp$ we
have in  $(\StMod G)_{\le\fp}$
\[
\Delta_G(\alpha)^\perp= (\StMod G)_{<\fp}={^\perp\Delta_G(\alpha)}.
\]

There is an analogy between point modules over finite group
schemes and standard objects of highest weight categories. In fact,
the analogy includes costandard objects, depending on whether one
thinks of a point module as induced or coinduced from a trivial
representation; see Theorem~\ref{th:gorenstein}. 

\begin{remark}
Let $\sfA$ be a highest weight category \cite{Cline/Parshall/Scott:1988a} with partially ordered set of weights $\Lambda$, which is assumed to be finite for simplicity. Thus $\sfA$ is an abelian length category with simple objects $\{L(\lambda)\}_{\lambda\in\Lambda}$. Now fix $\lambda\in\Lambda$ and consider the full subcategory $\sfA_{\le\lambda}$ of objects in $\sfA$ that have composition factors $L(\mu)$ with
$\mu\le\lambda$. The standard object $\Delta(\lambda)$ is a projective cover of $L(\lambda)$ in $\sfA_{\le\lambda}$ and its endomorphism ring is a division ring which we denote $K_\lambda$. This situation gives rise to the following recollement \cite[Theorem~3.9]{Cline/Parshall/Scott:1988a}.
\[
\begin{tikzcd}
  \sfA_{<\lambda}\arrow[tail]{rrr} &&&\sfA_{\le\lambda}
  \arrow[twoheadrightarrow,yshift=-1.5ex]{lll}
  \arrow[twoheadrightarrow,yshift=1.5ex]{lll}
  \arrow[twoheadrightarrow]{rrr}[description]{\Hom(\Delta(\lambda),-)} &&&\mod K_\lambda
  \arrow[tail,yshift=-1.5ex]{lll}
  \arrow[tail,yshift=1.5ex]{lll}
\end{tikzcd}
\]
Note that  $\Delta(\lambda)^\perp=\sfA_{<\lambda}={^\perp \nabla(\lambda)}$, where $\nabla(\lambda)$ denotes the costandard
 object corresponding to $\lambda$, namely, the injective envelope of $L(\lambda)$ in $\sfA_{\le\lambda}$.
\end{remark}

\end{document}